\documentclass{amsart}
\pdfoutput=1
\usepackage{url}
\usepackage{microtype}
\usepackage{tikz}
\usepackage[pdftex,colorlinks,citecolor=black,linkcolor=black,urlcolor=black,bookmarks=false]{hyperref}

\usepackage{doi}
\usepackage{booktabs}
\usepackage{amssymb}

\usepackage{color}
\usepackage[yyyymmdd,hhmmss]{datetime}

\theoremstyle{plain}
\newtheorem{theorem}{Theorem}

\newtheorem{lemma}[theorem]{Lemma}
\newtheorem{proposition}[theorem]{Proposition}

\newtheorem{question}[theorem]{Question}

\theoremstyle{remark}
\newtheorem{remark}[theorem]{Remark}

\numberwithin{theorem}{section}

\numberwithin{equation}{section}

\newcommand{\Q}{{\mathbb Q}}

\newcommand{\F}{\mathbb{F}}
\newcommand{\Pj}{\mathbb{P}}
\newcommand{\D}{{\mathcal{D}}}
\newcommand{\Ls}{{\mathcal{L}}}

\newcommand{\Ss}{{\mathcal{S}}}

\newcommand{\Cs}{{\mathcal{C}}}

\DeclareMathOperator{\Tr}{Tr\,}

\title{On the linear bounds on genera of pointless hyperelliptic curves}

\author[Pogildiakov]{Ivan Pogildiakov}
\address{
Ivan Pogildiakov
\newline \indent
Laboratoire GAATI, Universit\'e de la Polyn\'esie fran\c caise
\newline \indent
BP 6570 --- 98702 Faa'a, Tahiti, Polyn\'esie fran\c caise
} 
\email{ivan.pogildiakov@gmail.com}

\begin{document}

\begin{abstract} 
An irreducible smooth projective curve over $\F_q$ is called \emph{pointless} if it has no $\F_q$-rational points. In this paper we study the lower existence bound on the genus of such a curve over a fixed finite field $\F_q$. Using some explicit constructions of hyperelliptic curves, we establish two new bounds that depend linearly on the number $q$. In the case of odd characteristic this improves upon a result of R.~Becker and D.~Glass. We also provide a similar new bound when $q$ is even.
\end{abstract}

\maketitle

\section{Introduction}

Given a prime power $q=p^n$ it is natural to ask whether a non-singular curve over the finite field $\F_{q}$ having no rational points exists. For example, if $p>3$, then the Fermat curve $X^{q-1}+Y^{q-1}+Z^{q-1}=0$ serves an example of such a curve, since $\alpha^{q-1} \in \{0,1\}$ for every $\alpha$ in $\F_q$.
A positive answer to this question in the case of arbitrary characteristic follows from a result due to N.~Anbar and H.~Stichtenoth \cite[Theorem $1.2$]{Anbar13}. They show that, given $q$ and a non-negative number $N$, there is a~suitable function field defining a smooth curve over $\F_q$ of large genus having exactly $N$ rational points. Thus, given $q$, it is natural to ask what genus a pointless curve over $\F_q$ may have.

Let $\Cs$ be an irreducible non-singular projective (we will omit these words further) genus $g$ curve defined over $\F_q$. Denote by $N_k(\Cs)$ the number of rational points of $\Cs\otimes_{\F_q}\F_{q^k}$. The curve $\Cs$ can not have too many or too few rational points. More precisely, the Hasse--Weil--Serre bound on $N_k(\Cs)$ holds$\colon$ 
\[
q^k+1-g\left\lfloor 2q^{k/2} \right\rfloor \leq N_k(\Cs) \leq q^k+1+g\left\lfloor 2q^{k/2} \right\rfloor.
\]
Therefore, if a curve $\Cs$ has no points over $\F_{q^k}$, i.e. $N_k(\Cs)=0$, then the lower bound implies a~restriction on the genus$\colon$
\begin{equation}\label{eq:general_bound}
g \geq (q^k+1)/\left\lfloor 2q^{k/2} \right\rfloor.
\end{equation}
Hence the following question arises.

\begin{question}\label{qn:zero}
Given a prime power $q$ and an integer $g$ satisfying $\eqref{eq:general_bound}$, does there exist a non-singular genus $g$ curve over $\F_q$ having no $\F_{q^k}$-points?
\end{question}

A complete answer to this question seems to be very difficult to obtain. It is directly related to the question about the attainability of the lower Hasse--Weil--Serre bound, becoming especially difficult, when the lower bound is non-positive but close to zero.

In this paper we study the following special version of Question $\ref{qn:zero}$.

\begin{question}\label{qn:lite_version}
Given a prime power $q$, what is the minimal number $g_q^{min}$ such that for any $g \geq g_q^{min}$ there is a non-singular pointless genus $g$ curve over $\F_q$?
\end{question}

We restrict our consideration of these questions to hyperelliptic curves. Note that for these curves there are several known improvements of the Hasse--Weil--Serre bounds. They are mostly based on Stepanov's method. One of these improvements belongs to H.~M.~Stark. Applying one of his estimates on the number of rational points on a hyperelliptic curve to the case $g=2$ and $q=13$, he showed that every genus two curve $\Cs$ over $\F_{13}$ has $N_1(\Cs)\geq 2$ \cite[p.~288]{Stark72} (note that if a genus $2$ curve over $\F_q$ has no rational points, then the inequality $\eqref{eq:general_bound}$ implies that $q \leq 13$). This fact probably initiated the study of Question $\ref{qn:zero}$.

There are several known approaches to Question $\ref{qn:zero}$. Each of them involves different methods and constructions. One can identify three main directions$\colon$ when $q$ is fixed, when $g$ is fixed, and when $q$ and $g$ both vary. 

The asymptotic behavior of the minimal degree of smooth plane curves over $\F_q$ having no rational points when $q\to\infty$ was studied by S.~Yekhanin in \cite{Yekhanin07}. For a field $\F_q$ of characteristic $p>3$ a sequence of pointless non-singular plane curves over $\F_{q^{4^{k+1}}}$ of genus $g_k$ such that $\log_{q^{4^{k+1}}}{g_k}$ tends to $2/3$ was constructed. Namely, he showed that the Fermat curves
\[
X^{d_k}+Y^{d_k}+Z^{d_k}=1,
\]
where $d_k=2^{k+1}(q^2-1)\prod_{i=1}^{k}{(q^{4^i}+1)}$, have no points over $\F_{q^{4^{k+1}}}$. 

Aside from the trivial case of genus one, the question is completely resolved only in the genus two, three and four cases. As we mentioned above, the genus two case was treated by H.~M.~Stark in \cite{Stark72}, where he modified Stepanov's method to estimate the number of rational points on hyperelliptic curves over $\F_p$. Later on, as a by-product of their research concerning abelian surfaces, D.~Maisner and E.~Nart listed in \cite[Table 4]{Maisner02} all the pointless genus two curves over $\F_q$, $q<13$, up to isomorphism over $\F_q$. The genus three and four cases were investigated in detail by E.~W.~Howe, K.~E.~Lauter and J.~Top in \cite{Howe05}. For $g$ equal to $3$ or $4$ and every $q$, they either provide an example of a pointless genus $g$ curve over $\F_q$, or prove by various methods that such a curve does not exist. Question $\ref{qn:zero}$ still remains open if the genus is greater than $4$.

The case when $q$ is fixed has been studied more extensively. It can be subdivided into at least two problems.

The first one is to estimate the minimal genus of a curve over $\F_q$ having no points of degree $n>0$ or less. It was proven by P.~Clark and N.~Elkies (a citation can be found in \cite[p.~127]{Howe05}), that for every prime $p$ there is a constant $C_p$ such that for every integer $n>0$ there is a curve over $\F_p$ of genus at most $C_p n p^n$ that has no points of degree $n$ or less. 

This unpublished result was reproven by C.~Stirpe in her PhD thesis \cite{Stirpe11} by means of the class field theory. She showed \cite[Corollary 4.3]{Stirpe11} that, given a degree $n>n_0$ place $m$ of the rational function field $\F_q(x)$, there is a ray class field extension of $\F_q(x)$ with conductor $m$ and constant field $\F_q$ of genus
\[
\frac{(n-2)(q^n-q)}{2(q-1)},
\]
having no places of degree smaller than $n$. Note, the constant $n_0$ can be specified by \cite[Remark 4.5]{Stirpe11} as $n_0=\lfloor 6 \log_{q}{(9/4)} \rfloor < 14$. Furthermore, by taking a place $m$ of a special form, she refined this result as follows.

\begin{theorem}[Theorem 1.1 in \cite{Stirpe11}]\label{thm:stirpe}
For any prime $p$ there is a constant $C_p$ such that for any $n>0$ and for any power $q$ of $p$ there is a projective curve over $\F_q$ of genus $g\leq C_p q^n$ without points of degree strictly smaller than $n$.
\end{theorem}

However, a sharper bound is proven in the general case. She obtained that, given a prime $p$, there is a large integer $n_1$ such that for any $n>n_1$ there is a function field over $\F_p$ of genus $g\leq \frac{1}{2(p-1)} p^n$ without places of degree smaller than $n$. This bound was replaced by $g\leq C_p q^n$, $C_p=\frac{1}{2(p-1)}p^{n_1}$, for arbitrary $n$. 

The second problem is to provide a number $g_q$ such that for any $g \geq g_q$ there is a pointless genus $g$ curve over $\F_q$. Using Artin--Schreier extensions of function fields, N.~Anbar and H.~Stichtenoth showed in \cite[Theorem 1.2]{Anbar13} that for any $q$ there are positive constants $a_q$, defined implicitly, and $b_q=(p-1)(q+2(p-1)\sqrt{q}+3p-1)$ such that for any $g\geq g_{q,N}=a_q N + b_q$ there is a genus $g$ curve over $\F_q$ having exactly $N$ rational points. In particular, for any $g \geq g_q = g_{q,0}$ there is a pointless genus $g$ curve over $\F_q$. In the case when $q$ is a square the constant can be slightly improved, namely $g_q = 2q(p-1)+3p^2-2$, which follows from \cite[Remark 4.4]{Anbar13}.

We also remark that a more general result is implied by \cite[Theorem 1.4]{Anbar13}. Given a prime power $q$ and a positive integer $n$, there is a constant $g_0$ such that for every $g \geq g_0$ there is a curve over $\F_q$ having no points of degree $n$ or less. This differs from the results of C.~Stirpe, since the bound on the genus is worse, however, the existence is valid for any $g \geq g_0$. 

Finally, in the case of odd characteristic there are results due to R.~Becker and D.~Glass \cite{Becker13} that inspired our own research in this direction. In their work an explicit construction of pointless hyperelliptic curves of a certain form defined over a finite field $\F_q$ was provided. Let us discuss some of their results.

Let $q$ be a power of an odd prime $p$. There are two types of existence bounds on the genus of a pointless curve over $\F_q$. The first one is implied by \cite[Theorem 1.2]{Becker13}, which can be reformulated in terms of pointless curves in the following manner. 

\begin{theorem}[\protect{Reformulation of \cite[Theorem 1.2]{Becker13}}]\label{thm:BG_result1}
Let $a$ be the least residue of $g\mod{p}$. Suppose that $g \geq (p-a-1)(q-1)$, if $a<p-1$, or $g \geq (p-2a-2)(q-1)/2$, if $0 \geq a \geq (p-3)/2$. Then there exists a non-singular hyperelliptic pointless curve of genus $g$ defined over $\F_q$.
\end{theorem}

These bounds depend on the congruence class of the genus modulo $p$ and, in general, are non-linear. Furthermore, they showed \cite[Lemma 2.2]{Becker13} that if $a=p-1$ and $g \geq (q-1)/2$, then such a curve also exists. Putting the above together, one obtains the bound $g_q^{min} \leq (p-1)(q-1)$, which holds without any restriction on the genus.

R.~Becker and D.~Glass also established several results of another type, providing linear existence bounds. These bounds require imposing certain restrictive conditions on the genus. For example, they proved$\colon$

\begin{theorem}[\protect{\cite[Corollary 2.5]{Becker13}}]\label{thm:BG_result2}
Let $q$ be a power of a prime $p>2$. For a given $g\geq\frac{q-1}{2}$, set $d=\gcd(2g+2,q-1)$. If $(2g+2)^d \not\equiv (2g+3)^d \mod{p}$, then there exists a hyperelliptic curve over $\F_q$ of genus $g$ with no rational points.
\end{theorem}

Moreover, in the case of prime $q$ this theorem can be simplified as follows. Suppose $g$ is an integer such that $g\geq(p-3)/2$ and $g+1$ is prime to $(p-1)/2$. It results from \cite[Theorem 1.3]{Becker13} that there is a non-singular pointless genus $g$ curve over $\F_p$.

However, no linear bound without any restriction on $g$ was previously known, unless $q$ is a Sophie Germain prime. 

In this paper we provide a linear existence bound on the genus of a non-singular pointless curve defined over a fixed finite field $\F_q$ with no additional assumptions. Our main result is the following theorem$\colon$

\begin{theorem}\label{thm:linear_bound}
Let $q$ be a prime power. Set
\[
g_q = 
\begin{cases}
\max{\left\{ \frac{q-3}{2},2 \right\}}, & \text{$q$ is odd}, \\
\max{\left\{ q-1,2 \right\}}, & \text{$q$ is even}.
\end{cases}
\]
Suppose that $g \geq g_q$. Then there is a smooth genus $g$ hyperelliptic curve over $\F_q$ having no $\F_q$-points. 
\end{theorem}

This result is stronger than that of R.~Becker and D.~Glass \cite{Becker13}, since it provides unconditional bounds, which depend linearly on $q$. Furthermore, it covers the case of even characteristic. 

We finish this introduction with a brief discussion of Question $\ref{qn:lite_version}$. It can be completely answered for small values of $q$. We know that a pointless genus $2$ curve over $\F_q$ exists if and only if $q<13$. It is also known that for $q \leq 13$ there are pointless curves over $\F_q$ of genus $3$ and of genus $4$. As a corollary of Theorem $\ref{thm:linear_bound}$, for an arbitrary $q$ we have the bound 
\[
g_q^{min} \leq g_q = 
\begin{cases}
2,   & q=2, 3, 5, 7,\\
q-1, & \text{$q$ is even}, q \geq 4,\\
(q-3)/2, & \text{$q$ is odd}, q \geq 9.
\end{cases}
\] 
Thus, collecting all this together, we get $g_q^{min} = 2$ when $q < 13$ and $q \not= 8$, and $g_{13}^{min} = 3$. 

However, if $q=8$ or $q>13$, then we can obtain only the upper and the lower bounds on $g_q^{min}$. For example, there is a curve over $\F_2$ of genus $4$ having no points of degree $3$ or less (see the table in the end of \cite[Chapter $5$]{Stirpe11}). Therefore, we have $4 \leq g_8^{min} \leq 7$ and we do not know, whether a pointless curve over $\F_8$ of genus $5$ or $6$ exists. In general, one can estimate the quantity $g_{q}^{min}$ by means of the inequality $\eqref{eq:general_bound}$ and the bound above.

In Section $\ref{sec:preliminaries}$ we give some necessary preliminaries. In section $\ref{sec:the_proof_odd}$ we prove our result in the case of odd characterictic in several steps by explicit constructions of suitable smooth hyperelliptic curves. In this case the idea of finding such curves is similar to the one used in \cite{Becker13}. However, we prove their smoothness by an alternative method, which leads to an improvement of the results in \cite{Becker13}. In section $\ref{sec:the_proof_even}$ we finish the proof of Theorem $\ref{thm:linear_bound}$ providing an explicit construction of pointless curves over a finite field of even characteristic.

\textbf{Acknowledgments.} I would like to thank Prof. Alexey Zykin, my supervisor, for the support motivating me in this research, for careful reading of the earlier drafts of the paper and many helpful comments.

\section{Preliminaries}
\label{sec:preliminaries}

In this section we introduce the notation and the definitions, that will be used in the rest of the paper. All necessary general information on hyperelliptic curves over finite fields can be found in \cite{Liu02} (we will use Proposition $4.24$ and Remark $4.25$).

Let $p$ be a prime and let $q$ be a power of $p$. A genus $g$ hyperelliptic curve $\Cs$ over the finite field $\F_q$ can be defined by
\[
y^2 + Q(x) y = P(x),
\]
where $Q(x)$ and $P(x)$ are polynomials in $\F_q[x]$ satisfying
\[
2g+1 \leq \max\{ 2 \deg{Q(x)}, \deg{P(x)} \} \leq 2g+2.
\]
If $p$ is odd, then one can take $Q(x)=0$. In this case $\Cs$ is a smooth projective curve if and only if $P(x)$ has no repeated roots in $\overline{\F_q}$. In the case of even characteristic the smoothness of $\Cs$ is determined by a condition that $Q^{\prime}(x)^2 P(x) + P^{\prime}(x)^2$ and $Q(x)$ are coprime.

The curve $\Cs$ is the union of two affine curves 
\[
y^2 + Q(x) y = P(x) \text{ and } y^2 + x^{g+1} Q(1/x) y = x^{2g+2} P(1/x).
\]
The set of its rational points is the disjoint union of the set of rational points of the first affine curve and the set of rational points of the second one having $x=0$.

A hyperelliptic curve over $\F_q$ is said to be \emph{$\F_{q^k}$-maximal}, if it has precisely $2q^k+2$ points over $\F_{q^k}$. This definition is reasonable, since a hyperelliptic curve is a degree $2$ branched covering of the projective line, so that it has at most $2$ different points over each point of $\Pj^1$. 

We start with the following auxiliary lemma allowing us, in the case of odd characteristic, to search for $\F_q$-maximal non-singular hyperelliptic curves over $\F_q$ rather than pointless ones.

\begin{lemma}\label{lm:maximal_pointless}
A smooth $\F_q$-maximal genus $g$ hyperelliptic curve over $\F_q$ exists if and only if a smooth pointless hyperelliptic genus $g$ curve over $\F_q$ does.
\end{lemma}

\begin{proof}
Let $\Cs$ be a smooth $\F_q$-maximal genus $g$ hyperelliptic curve over $\F_q$. The Weil theorem implies that
\[
\Tr(Fr_{\Cs})=1+q-N_1(\Cs)=-1-q,
\]
where $\Tr(Fr_{\Cs})$ stands for the trace of the Frobenuis endomorphism of $H^1 (\Cs,\Q_l)$.
  
Let $\Cs^{\prime}$ be the $\F_q$-quadratic twist of the curve $\Cs$, then $\Tr(Fr_{\Cs^{\prime}})=-\Tr(Fr_{\Cs})$ and
\[
N_1(\Cs^{\prime})=1+q-\Tr(Fr_{\Cs^{\prime}})=0.
\]
Since the $\F_q$-quadratic twist of $\Cs^{\prime}$ is isomorphic over $\F_q$ to $\Cs$, the lemma is proven.
\end{proof}

Let $q$ be an odd prime power. In this case the method of finding pointless curves we use is similar to the one used in \cite{Becker13} and is as follows. We denote by $\chi_q(x)$ the quadratic character of $\F_q$. Let $F(x)$ be a polynomial over $\F_q$ such that
\begin{enumerate}
\item $F(x)$ is monic and square free of degree $2g+2$, $g>1$, \label{cnd:first}
\item $\chi_q(F(x))=1$ for any $x\in\F_q$.\label{cnd:second}
\end{enumerate}
Denote by $\Cs$ the hyperelliptic curve over $\F_q$ with an affine model $y^2=F(x)$. It follows from $\eqref{cnd:first}$ that $\Cs$ is a non-singular curve of genus $g$. Since $F(x)$ satisfies $\eqref{cnd:second}$, we see that the number of affine $\F_{q}$-points of the curve is $2q$. Since all the points of $\Cs$ at infinity satisfy $y^2=x^{2g+2} F(x^{-1})$ with $x=0$ and $F(x)$ is monic, $\Cs$ has two points at infinity. Thus the curve $\Cs$ is $\F_{q}$-maximal, and by Lemma $\ref{lm:maximal_pointless}$ there exists a non-singular hyperellptic curve of genus $g$ over $\F_q$ having no rational points. The main obstacle to applying the method is finding a polynomial over $\F_q$ satisfying the conditions $\eqref{cnd:first}$ and $\eqref{cnd:second}$. 

We begin with the study of the family of hyperelliptic curves $y^2=f_{g,l,a}(x)$, where
\begin{equation}
\label{eq:curves_type_1}
f_{g,l,a}(x) = x^{2g+2} - x^{2g+2-l(q-1)} + a^2.
\end{equation}
We denote by $\Ls(q,g)$ the quantity $\left\lfloor(2g+2)/(q-1)\right\rfloor$. Given an integer $g$ and $a$ an element of $\F_q$, the number of all the polynomials of the form $\eqref{eq:curves_type_1}$ is precisely $\Ls(q,g)$. We will also need the parameter $\D(q,g)=\gcd(2g+2,q-1)$ in the next section.

This family includes most of the curves considered in \cite{Becker13}. The polynomial $f_{g,l,a}(x)$, obviously, satisfies $\eqref{cnd:second}$ when $a\not=0$. In order to check $\eqref{cnd:first}$, we use the following lemma as a central tool$\colon$

\begin{lemma}\label{lmm:discriminant}
A polynomial $x^n - x^{n-m} + a$ has multiple roots if and only if
\[
n^N a^M - m^M (n-m)^{N-M}  = 0,
\]
where $d=\gcd(n,m)$, $N=n/d$, $M=m/d$.
\end{lemma}

This lemma can be regarded as a direct application of \cite[Theorem $4$]{Greenfield84}, for example. It claims that one can compute the discriminant of a trinomial $x^n + a x^k + b$ as
\[
(-1)^{n(n-1)/2} b^{k-1}\left[n^N b^{N-K} - (-1)^N (n-k)^{N-K} k^K a^N \right]^d,
\]
where $d=\gcd(n,k)$, $N=n/d$, $K=k/d$.

According to Lemma $\ref{lmm:discriminant}$, the polynomial $f_{g,l,a}(x)$ satisfies $\eqref{cnd:first}$ if and only if $g>1$ and the quantity
\begin{equation}
\label{eq:s_definition}
s_{g,l,a} = (2g+2)^N a^{2M} - (-l)^M (2g+2+l)^{N-M}
\end{equation}
is different from zero. Here $d=\gcd(2g+2,l(q-1))$, $N=n/d$ and $M=l(q-1)/d$.

\begin{remark}
The last lemma immediately implies that, if $q$ is odd, then, given an integer $g \geq g_q^1 = p(q-1)/2-1$, there is a non-singular genus $g$ curve over $\F_p$ having no $\F_q$-points. This estimate is of the same order as the one in Theorem $\ref{thm:BG_result1}$. Let us prove it now.

We want to show that there are $a\in\F_q^*$ and $1\leq l\leq\Ls(q,g)$ such that the curve $y^2=f_{g,l,a}(x)$ is an $\F_q$-maximal smooth hyperelliptic curve over $\F_p$ of genus $g$.

Choose an element of $\F_p^*$, say $a$. If $p \mid 2g+2$, then $s_{g,l_0,a}\not=0$, when $l_0=1$. Assume that $p$ does not divide $2g+2$. Since $g\geq g_q^1$, we obtain $\Ls(q,g)\geq p$. Hence we can choose $1\leq l_0 \leq \Ls(q,g)$ such that $l_0 \equiv -(2g+2)\mod{p}$. It follows that $s_{g,l_0,a}\not=0$. Thus, for any $g\geq g_q^1$ there is a number $l_0$ such that $1\leq l_0 \leq \Ls(q,g)$ and the polynomial $f_{g,l_0,a}(x)$ is square free by Lemma $\ref{lmm:discriminant}$. 

It remains to apply Lemma $\ref{lm:maximal_pointless}$.
\end{remark}

We also note that this construction and some of those from the next section are related to the question of the existence of curves having no points of degree greater than one.

\section{Proof of Theorem $\ref{thm:linear_bound}$, the case of odd characteristic}
\label{sec:the_proof_odd}

Along the section $q$ stands for a power of an odd prime $p$. For such $q$ the proof of Theorem $\ref{thm:linear_bound}$ consists of several parts depending on the values of the parameters $p$, $q$, $g$, $\D(q,g)$, and $\Ls(q,g)$. We first treat the easy case, when $\D(q,g)>2$.

\begin{proposition}\label{prop:gcd_is_over_two} 
Let $g$ be such that $g\geq \frac{q-1}{2}$ and $\D(q,g)>2$. Then there is a smooth $\F_q$-maximal hyperelliptic curve over $\F_q$ of genus $g$ of the form $y^2 = f_{g,l,a}(x)$.
\end{proposition}

\begin{proof}
Let us take $l=1$ and suppose that for any $a\in\F_q^*$ the quantity $s_{g,l,a}$ is equal to zero. Then the value of
\[
a^{\frac{2(q-1)}{\D(q,g)}}=(-1)^{\frac{q-1}{\D(q,g)}}(2g+3)^{\frac{2g+2-(q-1)}{\D(q,g)}} (2g+2)^{-\frac{2g+2}{\D(q,g)}}
\]
does not depend on the choice of $a$. Since this identity holds for $a=1$, we see that $a^{2(q-1)/\D(q,g)}$ is equal to $1$ for every $a\in\F_q^*$. It means that the order of every element of $\F_q^*$ divides $2(q-1)/\D(q,g)<q-1$, which is a contradiction. 

Hence there is some $a\in\F_q^*$ such that the polynomial $f_{g,1,a}(x)$ has no multiple roots and the curve $y^2=f_{g,1,a}(x)$ is the desired one.
\end{proof}

We would like to draw attention to the similarity between the following technical lemma and Theorem $\ref{thm:BG_result2}$.

\begin{lemma}\label{lmm:system}
Let $a\in\F_q^*$. Suppose that for some $g\geq \max\left\{\frac{q-1}{2},2\right\}$ the quantities $s_{g,l,a}$ are equal to $0$ for all $1\leq l\leq \Ls(q,g)$; then the equations $(1+lx)^{\D(q,g)}\equiv1\mod{p}$, $1\leq l\leq \Ls(q,g)$, have a common non-zero solution in $\F_p^*$.
\end{lemma}

\begin{proof}
Since $s_{g,l,a}=0$ for all $l$'s, it follows from $\eqref{eq:s_definition}$ that $2g+2 \not \equiv 0 \mod{p}$, and for each $l$ there is an identity
\[
(2g+2)^{\frac{2g+2}{d}} a^{2\frac{l(q-1)}{d}} \equiv (-l)^{\frac{l(q-1)}{d}} (2g+2+l)^{\frac{2g+2-l(q-1)}{d}} \mod{p},
\]
where $d=\gcd(2g+2,l(q-1))$. After rising to power $d$ both of its sides and some evident transformation, we get
\[
(1+(2g+2)^{-1}l)^{2g+2} \equiv 1 \mod{p},
\]
for each $l$, and since $1+(2g+2)^{-1}l$ is an element of $\F_q^*$, it is equivalent to
\[
(1+(2g+2)^{-1}l)^{\gcd(2g+2,q-1)}=(1+(2g+2)^{-1}l)^{\D(g,q)}\equiv 1\mod{p},
\]
for every $1\leq l\leq \Ls(q,g)$.
\end{proof}

In the case of odd $p$ we show below that if $g\geq \max\left\{(q-1)/2,2\right\}$, $\D(q,g)=2$ or $g = (q-3)/2$, $p>5$, then there is a non-singular $\F_q$-maximal hyperelliptic genus $g$ curve defined over the prime field $\F_p$. The proposition below shows that if the parameter $\D(q,g)$ takes the minimal possible value and $\Ls(q,g)\geq 2$, then one can choose a suitable $1\leq l\leq 2$.

\begin{proposition}\label{prop:L_is_over_two}
Let $g$ be an integer such that $g\geq \max\left\{\frac{q-1}{2},2\right\}$, let $\D(q,g)=2$ and $\Ls(q,g)\geq 2$. Then there is a smooth $\F_q$-maximal hyperelliptic curve over $\F_p$ of genus $g$ of the form $y^2 = f_{g,l,a}(x)$.
\end{proposition}

\begin{proof} 
Chose any $a\in\F_p^*$. Suppose that $s_{g,l,a}=0$ for all $1\leq l\leq \Ls(q,g)$. According to Lemma $\ref{lmm:system}$, there is a system of at least two equations$\colon$
\[
(1+x)^2=1, \quad (1+2x)^2=1, \quad \ldots
\] 
which has a solution in $\F_p^*$, say $x_0$. It is obvious that $x_0\equiv -2\mod{p}$ is the unique solution of the first equation, but $x_0$ does not satisfy the second equation. It means that the system has no solutions in $\F_p^*$, and $s_{g,l,a}\not=0$ for some $l$. It is clear that one can let either $l=1$ or $l=2$. Hence it is possible to construct a curve with desired properties, using either $f_{g,1,a}(x)$ or $f_{g,2,a}(x)$.
\end{proof}

\begin{remark} 
As a direct corollary of Propositions $\ref{prop:gcd_is_over_two}$, $\ref{prop:L_is_over_two}$ and Lemma $\ref{lm:maximal_pointless}$ we obtain a linear bound on the genus as follows. Suppose $g\geq \max{ \left\{ q-2,2 \right\} }$; then $\Ls(q,g) \geq 2$, so there is a smooth pointless genus $g$ curve over $\F_q$. This refines \cite[Corollary 2.9]{Becker13}. However, our goal is to prove a stronger bound.
\end{remark}

The only possibility that remains to be treated is $\Ls(q,g)=1$, $D(q,g)=2$.

\begin{proposition}\label{prop:old_curve}
Let $g$ be such that $g\geq \max\left\{(q-1)/2,2\right\}$, $D(q,g)=2$, and $\Ls(q,g)=1$. Assume that at least one of the following conditions holds$\colon$
\begin{enumerate}
\item $q$ is a square,
\item $q=p^{2n+1}$ and $p\equiv 1 \mod{8}$,
\item $g$ is not of the form $\frac{p(2k+1)-5}{4}$, $k\geq 0$.
\end{enumerate}
Then there is a non-singular $\F_q$-maximal hyperelliptic curve over $\F_p$ of genus $g$ of the form $y^2 = f_{g,l,a}(x)$.
\end{proposition}

\begin{proof}
We want to show that under these assumptions $s_{g,1,a}\not\equiv 0\mod{p}$ for some $a\in\F_p^*$. If $g$ is such that $p$ divides $2g+2$, then $s_{g,1,a}\not\equiv 0\mod{p}$ by its definition. We can therefore assume that $2g+2\not\equiv 0\mod{p}$. In this case it follows from $\eqref{eq:s_definition}$ that
\begin{align*}
s_{g,1,a} & = (2g+2)^{N} a^{2M} - (-1)^{M} ((2g+2)+1)^{N-M}\\
 & = (2g+2)^{N} - (-1)^{M} (2g+2)^{N-M} \left(1+(2g+2)^{-1}\right)^{N-M},
\end{align*}
where $N=g+1$ and $M=(q-1)/2$. Thus the value of $a$ does not affect $s_{g,1,a}$, i.e. $s_{g,1,a}$ coincides with $s_{g,1,1}$ for any $a\in\F_q^*$.

Suppose that $s_{g,1,1}$ is zero; then $(1+(2g+2)^{-1})^2\equiv 1\mod{p}$ by Lemma $\ref{lmm:system}$. This implies $2g+2\equiv (-2)^{-1}\mod{p}$. Hence, if $g\not=\frac{p(2k+1)-5}{4}$ for any integer $k\geq 0$, then $s_{g,1,1}\not=0$.

Let us assume that $(1+(2g+2)^{-1})^2\equiv 1\mod{p}$. We put
\begin{align*}
\Ss =(-1)^{M+1} 2^{N} s_{g,1,1} & = (-1)^{M+1} 2^{N} \left( (-1)^N 2^{-N} - (-1)^{M} 2^{M-N} \right) \\
& = (-1)^{N+M+1} + 2^M.
\end{align*}
It equals zero simultaneously with $s_{g,1,1}$. Also, note that $N$ and $M$ can not be even at the same time, since 
\[
\gcd(N,M)=\gcd(g+1,(q-1)/2)=\D(q,g)/2=1.
\]

Consider several cases. If $q$ is a square, then every element of $\F_p$ is a quadratic residue in $\F_q$. Since $M=(q-1)/2$ is even, we see that $N+M+1$ is even. Thus $\Ss=2$ is not equal to zero, since $p>2$. Therefore, $s_{g,1,1}\not=0$.

Suppose $q$ is not a square, then $\Ss$ depends on the congruence class of $p$ modulo $8$. It is well known that every quadratic (non) residue in $\F_p$ remains one in $\F_q$. Using this fact and the fact that if $M$ is even, then $N$ must be odd, we find that
\[
\Ss = 2^{(q-1)/2}-(-1)^{(q-1)/2}(-1)^{g+1}=\begin{cases}
2, & p\equiv 1 \mod{8},\\
-1+(-1)^{g+1}, & p\equiv 3 \mod{8},\\
0, & p\equiv 5 \mod{8},\\
1+(-1)^{g+1}, & p\equiv 7 \mod{8}.
\end{cases}
\]
This formula immediately implies that $s_{g,1,1}\not=0$ when $p\equiv 1\mod{8}$.
\end{proof}

Further study of curves of type $\eqref{eq:curves_type_1}$ is not fruitful. We finish our proof of Theorem $\ref{thm:linear_bound}$ in the case of odd characteristic taking into account hyperelliptic curves of different forms. 

However, let us first assume that $q$ is prime. This assumption allows to consider a special type of curves which is easier to work with. In particular, applying Lemma~$\ref{lm:maximal_pointless}$ to the following result, one can omit the restriction in \cite[Theorem 1.3]{Becker13}.

\begin{theorem}\label{thm:prime_field}
Let $g$ be an integer such that $g \geq \max\left\{\frac{p-1}{2},2\right\}$. Then there exists a non-singular $\F_p$-maximal hyperelliptic genus $g$ curve over $\F_p$ of one of the following forms$\colon$ 
\[
y^2 = f_{g,l,a}(x) \text{ or } y^2=x^{p+(p-1)/2}+x^{p-1}+x^{(p-3)/2}+1.
\]
\end{theorem}

\begin{proof}
If $\D(p,g)>2$ or $\Ls(q,g)>1$, then Propositions $\ref{prop:gcd_is_over_two}$ and $\ref{prop:L_is_over_two}$ can be applied. Assume that $\D(p,g)=2$ and $\Ls(q,g)=1$. Then 
\begin{equation}\label{eq:g_bound}
\frac{p-1}{2}\leq g < p-2.
\end{equation}
Let $g$ be an integer such that $g=(p(2k+1)-5)/4$ for some $k\geq0$. It follows from the formula $\eqref{eq:g_bound}$ that $k=1$, i.e. $g=(3p-5)/4$. Note that $\deg{F(x)}=2g+2$.

Suppose that $p\equiv 5\mod{8}$; then the number $g$ is not integer. Thus Proposition~$\ref{prop:old_curve}$ can be applied. 

We are left with the case $p\equiv -1 \mod{4}$. We can assume that $p>3$. Consider a curve $y^2=F(x)$ over $\F_p$, where $F(x)=x^{p+(p-1)/2}+x^{p-1}+x^{(p-3)/2}+1$. Let us show that $F(x)$ has no multiple roots in $\overline{\F_p}$. We first note that the following identities hold$\colon$ 
\begin{eqnarray*}
F(x)+2x F^{\prime}(x)              & =\ r_1(x) & =\ -x^{p-1}-2x^{(p-3)/2}+1,\\ 
2F^{\prime}(x) -x^{(p-1)/2} r_1(x) & =\ r_2(x) & =\ -x^{(p-5)/2}(x^2+3).
\end{eqnarray*}
Now, suppose that $F(x)$ has a multiple root, say $\alpha$. Then we see that $\alpha^2 = -3$, since $r_2(\alpha)=0$ and $\alpha\not=0$. We also obtain
\[
-\alpha^2 r_1(\alpha)=\left(\alpha^{(p+1)/2}\right)^2+2\alpha^{(p+1)/2}+3=0.
\] 
This equation implies $\alpha^{(p+1)/2}=-1\pm\sqrt{-2}$. Furthermore, we can compute $\left( \alpha^{(p+1)/2} \right)^4$ in two different ways$\colon$
\begin{align*}
\left( \alpha^{(p+1)/2} \right)^4 & = (-1\pm\sqrt{-2})^4 = -7 \pm 4\sqrt{-2} \\
& = \left( \alpha^2 \right)^{p+1} = 9.
\end{align*}
Hence the equality $\sqrt{-2} = 4$ holds for the field $\F_p$, but this is not possible whenever $p>3$. Therefore, we have a contradiction and the polynomial $F(x)$ satisfies $\eqref{cnd:first}$. 

The polynomial $F(x)$ also satisfies $\eqref{cnd:second}$. Indeed, for any $x\in\F_p^*$ we have
\[
F(x)=\chi_p(x)x+2+\chi_p(x)x^{-1}=\chi_p(x)x^{-1}\left(1+\chi_p(x)x\right)^2.
\] 
Since $\chi_p(-1)=-1$, it follows that $\chi_p(\chi_p(x)x^{-1})=1$ for any $x\in\F_p^*$. Regardless of the value of $\chi_p(x)$, we obtain $\chi_p(F(x))=1$ for any $x\in\F_p$. 

Thus, the curve $y^2=F(x)$ is a non-singular $\F_p$-maximal genus $g$ hyperelliptic curve over $\F_p$.
\end{proof}

\begin{remark}
By the same method one can show that if $q$ is not a prime such that $q\equiv -1 \mod{4}$, then the curve $y^2=x^{q+(q-1)/2}+x^{q-1}+x^{(q-3)/2}+1$ is smooth and $\F_q$-maximal. However, this construction does not allow to obtain the desired bound on the genus.
\end{remark}

In order to get the bound $g \geq \max{\left\{\frac{q-1}{2}, 2\right\}}$ for odd $q$, we prove that if $q$ is neither a square nor a prime and $g$ is such that 
\begin{equation}\label{cnd:g_remained}
(q-1)/2\leq g<q-2,\ \gcd(2g+2,q-1)=2 \text{ and } 2g+2\equiv-1/2\mod{p},
\end{equation}
then there exists a non-singular genus $g$ pointless curve over $\F_q$. Now, we concentrate our attention only on this case. 

Consider the following family of hyperelliptic curves$\colon$
\begin{equation}\label{eq:family2}
y^2 = F_{n,b,\xi}(x)=x^{q-1+n} + b^2 x^{2n} - (2 b^2 \xi + 1) x^n + b^2 \xi^2.
\end{equation}
Let $b$, $\xi\in\F_p^*$, $\chi_p(\xi)=-1$, and let $n$ be an even number. Under these assumptions the polynomial $F_{n,b,\xi}(x)$ satisfies $\eqref{cnd:second}$. Indeed, we see that $\chi_q(F_{n,b,\xi}(x))=1$ for any $x\in\F_q$, since $F_{n,b,\xi}(x)=(x^{q-1}-1) x^n + b^2(x^n - \xi)^2$. Suppose that $\eqref{cnd:first}$ holds for $F_{n,b,\xi}(x)$. Then the hyperelliptic curve $\eqref{eq:family2}$ over $\F_p$ is smooth $\F_q$-maximal of genus 
\[
g=
\begin{cases}
(q-3+n)/2, & \text{$n$ is even and $0<n<q-1$,}\\
n-1, & n\geq q-1.
\end{cases}
\]

Recall that $g$ must satisfy $\eqref{cnd:g_remained}$, so we have the following restrictions on $n\colon$
\begin{equation}\label{cnd:n}
\text{$n$ even, $0<n<q-1$, $\gcd(n,q-1)=2$ and $n\equiv 1/2\mod{p}$.}
\end{equation}
Let us show that for such $n$ there are $b$ and $\xi$ in $\F_p$ such that $F_{n,b,\xi}(x)$ satisfies both $\eqref{cnd:first}$ and $\eqref{cnd:second}$. Our proof of this fact splits into two steps depending on $p$. 

\begin{lemma}\label{lm:small_p}
Suppose that $p$ is either $3$ or $5$. Let $n$ be such that $n\equiv1/2\mod{p}$. There are $b$, $\xi\in\F_p^*$, $\chi_p(\xi)=-1$, such that $F_{n,b,\xi}(x)$ has no multiple roots in $\overline{\F_p}$.
\end{lemma}

\begin{proof}
Case $p=3$. There is just one polynomial of the desired form over $\F_p\colon$
\[
F(x)=F_{n,\pm1,-1}(x)=x^{q-1+n} + x^{2n}  + x^n + 1.
\]

Let $\alpha\in\overline{\F_p}$ be a multiple root of $F(x)$. Since $n\equiv 2\mod{3}$, we have
\[
x F^{\prime}(x)= x^{q-1+n} + x^{2n}  - x^n.
\]
Note that $F(\alpha)-\alpha F^{\prime}(\alpha) = 2 \alpha^{n}+1=0,$ so $\alpha^n=1$. However,
\[
F(\alpha)=\alpha^{q-1+n} + \alpha^{2n}  + \alpha^n + 1=\alpha^{q-1}=0,
\]
so we get a contradiction. Thus $F(x)$ is square free.

Case $p=5$. One of the following polynomials is square free$\colon$
\[
F_{n,1,2}(x)=x^{q-1+n}+x^{2n}-1,\ F_{n,2,-2}(x)=x^{q-1+n}-x^{2n}+1.
\]
The Theorem $4$ in \cite{Greenfield84} (see Section $\ref{sec:preliminaries}$) implies that a trinomial $x^{q-1+n} \pm x^{2n} \mp 1$ is square free if and only if
\[
\Ss_{\pm} = (q-1+n)^N (\mp1)^{N-K} - (-1)^N (q-1-n)^{N-K} (2n)^K (\pm1)^N
\]
is non-zero, where $d=\gcd(q-1+n,2n)$, $N=(q-1+n)/d$ and $K=2n/d$. By the assumption on $n$ we have $n\equiv -2\mod{5}$ and
\[
\Ss_{+} = 2^N (-1)^{N-K} - (-1)^N = (-1)^{N-K}(2^N-(-1)^K),\ \Ss_{-} = 2^N - 1.
\]

Suppose $\Ss_-$ equals zero. Then, since $p=5$, we see that $4$ divides $N$. Note that $\gcd(N,K)=1$, thus $K$ is odd. Hence $\Ss_+\not=0$ and the lemma is proven.
\end{proof}

\begin{proposition}\label{prop:new_curve}
Let $q$ be not a square. Let $g$ be an integer satisfying $\eqref{cnd:g_remained}$. Then there exists a non-singular $\F_q$-maximal hyperelliptic curve over $\F_p$ of genus $g$ of the form $y^2 = F_{n,b,\xi}(x)$. 
\end{proposition} 

\begin{proof}
Take an integer $n$ satisfying $\eqref{cnd:n}$ such that $q-1+n=2g+2$. Take $b$, $\xi\in\F_p^*$, $\chi_p(\xi)=-1$. As we showed above, the polynomial $F_{n,b,\xi}(x)$ satisfies $\eqref{cnd:second}$. Our goal is to prove that $b$ can be chosen in a manner that $F_{n,b,\xi}(x)$ is a square free polynomial, i.e. satisfies $\eqref{cnd:first}$. According to Lemma $\ref{lm:small_p}$, we can assume that $p>5$.

Let $\alpha\in\overline{\F_q}$ be a multiple root of $F_{n,b,\xi}(x)=x^{q-1+n} + b^2 x^{2n} - (2 b^2 \xi + 1) x^n + b^2 \xi^2$. Then $\alpha$ is a root of the following polynomials (recall that $n\equiv 1/2\mod{p}$)$\colon$
\begin{align*}
-2 x F_{n,b,\xi}^{\prime}(x) & = x^{q - 1 + n} - 2 b^2 x^{2n} + (2 b^2 \xi + 1) x^n,\\
r(x) & = F_{n,b,\xi}(x) + 2 x F_{n,b,\xi}^{\prime}(x) = 3 b^2 x^{2n} - 2 (2 b^2 \xi + 1) x^n + b^2 \xi^2. 
\end{align*}
By evident computations we have
\begin{align}
\alpha^n & = \frac{2 b^2 \xi + 1 \pm \sqrt{b^4 \xi^2 + 4 b^2 \xi + 1} }{3 b^2}, \label{eq:power_n} \\
\alpha^{q-1} & = 2 b^2 \alpha^{n} - (2 b^2 \xi + 1) = -\frac{1}{3}(2 b^2 \xi + 1) \pm \frac{2}{3} \sqrt{b^4 \xi^2 + 4 b^2 \xi + 1}. \label{eq:power_q-1}
\end{align}

Let us show that $b$ can be chosen inside $\F_p^*$ in the way that 
\[
\chi_p(b^4 \xi^2 + 4 b^2 \xi + 1)=1.
\]
We need to show that there is such $s$ in $\F_p^*$ that 
\begin{equation}\label{eq:ell_curve}
s^2=b^4 \xi^2 + 4 b^2 \xi + 1.
\end{equation}
Since $p>5$, the equation  defines an elliptic curve over $\F_p$. It has two points at infinity, two points with $b=0$, and no more than two points with $s=0$. According to Hasse--Weil bound, the curve $\eqref{eq:ell_curve}$ has at least $p + 1 - \left \lfloor 2\sqrt{p} \right\rfloor$ rational points. Thus, when $p>11$, this curve has at least $7$ points. In the cases $p=7$ or $11$ one can let $\xi=-1$. Then the curve $\eqref{eq:ell_curve}$ has four points with $b = \pm 1/2$ and $s=\pm1/4$. Hence, given $\xi\in\F_p^*$ such that $\chi_p(\xi)=-1$, there is an admissible value for $b$ in $\F_p^*$.

Let $(b,s)$ be a rational point of the curve $\eqref{eq:ell_curve}$ such that $bs \not = 0$. Put $t=b^2 \xi$. Let us remark that 
\begin{equation}\label{eq:powers}
\alpha^n = \frac{(2t+1) \pm s}{3b^2},\ \alpha^{q-1} = \frac{-(2t+1) \pm 2s}{3}.
\end{equation}
are elements of $\F_p$. It easily follows that $\alpha^{\gcd(n,q-1)} = \alpha^{\gcd(2g+2,q-1)}=\alpha^2$ is in $\F_p^*$, and $\alpha^{q-1} = (\alpha^2)^{(q-1)/2}=\pm1$. Taking into account $\eqref{eq:ell_curve}$ and $\eqref{eq:powers}$ we obtain 
\[
4 s^2 = ((2t+1)+3\alpha^{q-1})^2 = \left((2t+1)\pm3\right)^2 = 4 t^2 + 4 t + 10 \pm (12t+6).
\]
By the assumption $(b,s)$ is a solution of $\eqref{eq:ell_curve}$, so that 
\[
4 s^2 = 4 t^2 + 4 t + 10 \pm (12t+6) = 4 t^2 + 16 t + 4.
\]
It follows that either $12\equiv0\mod{p}$ or $-24t\equiv0\mod{p}$. Thus we get a contradiction, since $p>5$ and $b\in\F_p^*$.
\end{proof} 

The last ingredient in the case of odd $q$ is the following result which refines \cite[Lemma $2.1$]{Becker13}.

\begin{proposition}\label{prop:last_genus}
Let $q>5$. There is an $\F_q$-maximal non-singular curve over $\F_q$ of genus $\frac{q-3}{2}$ of the form
\[
y^2 = x^{q-1} + 2 \frac{\alpha-\beta}{\alpha+\beta} x^{\frac{q-1}{2}} + 1,
\]
where $\alpha$, $\beta\in\F_q^*$ such that $\chi_q(\alpha)=\chi_q(\beta)=\chi_q(\alpha+\beta)=1$.
\end{proposition}

\begin{proof}
It is easy to see that if $q>5$, then there are three non-zero quadratic residues $\alpha$, $\beta$, $\gamma$ in $\F_q$ such that $\gamma = \alpha + \beta$. Indeed, let $\gamma=\delta^2$, $\delta \in \F_q^*$. It is well known that the curve $x^2+y^2=\gamma$ has $q+1 \geq 8$ rational points over $\F_q$. In particular, it has two points at infinity, and four affine points $(\pm\delta,0)$ and $(0,\pm\delta)$. Therefore, the curve has at least two affine points $(a,b)$ such that $ab \not= 0$. Thus, there are $\alpha$, $\beta \in \left(\F_q^*\right)^2$ such that $\alpha + \beta = \gamma$.

Let $F(x) = \frac{\alpha}{\gamma} (x^{\frac{q-1}{2}} + 1)^2  + \frac{\beta}{\gamma} (x^{\frac{q-1}{2}} - 1)^2$. It is easy to see that the conditions $\eqref{cnd:first}$ and $\eqref{cnd:second}$ hold for this polynomial. Indeed, we have 
\[
F(x) = \left(x^{\frac{q-1}{2}}\right)^2 + 2 \frac{\alpha-\beta}{\gamma} x^{\frac{q-1}{2}} + 1.
\]
It has a repeated root if and only if $\frac{\alpha-\beta}{\alpha+\beta} = \pm1$. Therefore, $F(x)$ satisfies $\eqref{cnd:first}$, since $\alpha$ and $\beta$ are both non-zero. For any $a\in\F_q^*$ we see that $F(a)$ is either $4 \alpha / \gamma$ or $4 \beta / \gamma$, which are both non-zero quadratic residues. Therefore, $\chi_q(F(a))=1$ for any $a\in\F_q$, so that $\eqref{cnd:second}$ holds.

As a consequence, if $q>5$, then a hyperelliptic curve $y^2 = F(x)$ over $\F_q$ is a smooth $\F_q$-maximal curve of genus $\frac{q-3}{2}$. Hence the result.
\end{proof}

\begin{remark}
Using the same arguments, one can prove that if $p>5$ and $q=p^n$, then there are $\alpha$, $\beta$, $\gamma$ in $\F_p^*$ such that $\gamma=\alpha+\beta$ and $\chi_p(\alpha)=\chi_p(\beta)=\chi_p(\gamma)=1$. Define $F(x)$ as in the last proposition. Then the curve $y^2 = F(x)$ is defined over $\F_p$, has no singular points and is $\F_q$-maximal. This construction gives an $\F_q$-maximal curve defined over the prime field $\F_p$, however, it requires the restriction $p>5$ on the characteristic of the field $\F_q$.
\end{remark}

Now, if $q$ is odd, then Propositions $\ref{prop:gcd_is_over_two}$, $\ref{prop:L_is_over_two}$, $\ref{prop:old_curve}$, $\ref{prop:new_curve}$ and $\ref{prop:last_genus}$ imply that for any $g \geq g_q$ (with the constant $g_q$ defined in the statement of Theorem $\ref{thm:linear_bound}$) there exists a non-singular $\F_q$-maximal genus $g$ hyperelliptic curve over $\F_q$. In order to finish the proof of Theorem $\ref{thm:linear_bound}$ in the case of odd characteristic, one applies Lemma $\ref{lm:maximal_pointless}$.

\section{Proof of Theorem $\ref{thm:linear_bound}$, the case of even characteristic}
\label{sec:the_proof_even}

Finally, it remains to prove the result in the case when $q$ is even. The existence of pointless curves over $\F_2$ is already known. It follows from the part $(i)$ of Proposition~$4.1$ in \cite{Stichtenoth11} that for every $g \geq 2$ there is a pointless smooth genus $g$ curve over $\F_2$. In other cases, when $q$ is even and $q>2$, we have the following result.

\begin{proposition}\label{prop:phc_char2}
Let $q=2^n$, $q>2$. For every $g \geq q-1$ there is a pointless non-singular genus $g$ curve over $\F_q$ of the form
\[
y^2 + a (x^{g+1} + x^{g+1-(q-1)} + c) y = b (x^{2g+2}+x^{2g+2-2(q-1)}+d)
\]
for some $a$, $b$, $c$, $d\in\F_q^*$.
\end{proposition}

\begin{proof}
Let $\Cs$ be the projective curve defined by the affine equation
\[
y^2 + a f(x) y = b h(x),
\] 
where $a$, $b\in\F_q^*$ and 
\begin{align*}
f(x) &= x^{g+1} + x^{g+1-(q-1)} + c, \\
h(x) &= x^{2g+2} + x^{2g+2-2(q-1)} + d.
\end{align*} 

Recall that a curve $y^2 + Q(x)y=P(x)$ over a finite field of characteristic $2$ is non-singular (as a projective curve) if only and only if the polynomials $Q(x)$ and $R(x)= Q^{\prime}(x)^2 P(x) + P^{\prime}(x)^2$ are coprime (see Section $\ref{sec:preliminaries}$).

Assume that $c$, $d$ in $\F_q^*$ are such that $c^2\not=d$. Then it is easy to see that the curve $\Cs$ is smooth. Indeed, we have $R(x)=a^2 b f^{\prime}(x)^2 h(x) + b^2 h^{\prime}(x)^2$. Note that $h^{\prime}(x)=0$ and $f^{\prime}(x)$ is a monomial, which is obviously prime to $f(x)$. Thus we have $\gcd(R(x), a f(x)) = \gcd(h(x), f(x))$. Since $h(x) = f(x)^2 + c^2 + d$, where $c^2+d\not=0$ (by the assumption), we see that $f(x)$ and $h(x)$ are coprime. Hence $\gcd(R(x), a f(x)) = 1$, so that $\Cs$ is a smooth curve of genus $g$ over $\F_q$. 

Remark that the choice of $a$ and $b$ does not affect the smoothness of the curve $\Cs$. We can therefore use this fact to obtain a pointless curve.

As it was explained in Section $\ref{sec:preliminaries}$, the curve $\Cs$ is the union of two affine curves
\[
y^2 + a f(x) y = b h(x) \text{ and } y^2 + a x^{g+1} f(1/x) = b x^{2g+2} h(1/x).
\]
The set of rational points of the curve $\Cs$ is the disjoint union of the set of rational points of the first affine curve and the set of rational points of the second one having $x=0$. Hence any rational point of $\Cs$ corresponds to a solution of either $y^2 + a c y + b d = 0$ or $y^2 + a y + b = 0$ over $\F_q$. These equations are, obviously, equivalent to $y^2 + y + bd (ac)^{-2} = 0$ and $y^2 + y + b a^{-2} = 0$ respectively. Recall that by the Hilbert'90 theorem the equation $y^2 + y + \alpha = 0$ has no solutions in $\F_q$ if and only if $\Tr_{\F_q/\F_2}{\alpha}=1$. Thus, the curve $\Cs$ has no rational points if and only if
\begin{equation}\label{trace_condition}
\Tr_{\F_q/\F_2}{\left(\frac{b}{a^2} \cdot \frac{d}{c^2}\right)}=1 \text{ and } \Tr_{\F_q/\F_2}{\left(\frac{b}{a^2}\right)}=1.
\end{equation}

Since $q>2$, there are two distinct elements in $\F_q$, say $\alpha$ and $\beta$, such that $\Tr_{\F_q/\F_2}{\alpha}=\Tr_{\F_q/\F_2}{\beta}=1$. Let $b = a^2 \alpha$ and $d = c^2 \frac{\beta}{\alpha}$. It is clear that $c^2 \not= d$ and the conditions $\eqref{trace_condition}$ hold. Thus $\Cs$ is a smooth curve of genus $g$ having no rational points.
\end{proof}

Now, the goal of the paper is achieved and Theorem $\ref{thm:linear_bound}$ is proven.

\begin{remark}
As we have already written, H.~Stichtenoth in \cite[Section 4]{Stichtenoth11} considered the case $q=2$. He showed that for any $g \geq 2$ a genus $g$ curve over $\F_2$
\[
y^2 + y = \frac{x^2+x}{f(x)} + 1,
\]
where $f(x)\in\F_2[x]$ is an irreducible polynomial of degree $g+1$, has no rational points. By slightly modifying this construction, N.~Anbar reproved this result in \cite[Section 5]{Anbar13}. Motivated by this, we note that a lot of examples of pointless curves can be easily obtained by means of the Artin--Schreier curves for arbitrary $q$. More precisely, their constructions can be generalized almost verbatim as follows.

Let $u(x)$ and $v(x)$ be monic polynomials of degree $n+1$ having no roots in $\F_q$. Suppose that $v(x)$ is irreducible. Then the Artin--Schreier curve 
\begin{equation}\label{eq:Art_Sch_curve}
y^q-y=\frac{u(x)}{v(x)}
\end{equation}
is a non-singular curve over $\F_q$. By \cite[Lemma 2.1]{Anbar13} it is has genus $(q-1)n$. This curve is pointless, since $\frac{u(x)}{v(x)}\not=0$ and $\frac{x^{n+1} u(1/x)}{x^{n+1} v(1/x)}\not=0$ for any $x\in\F_q$.

This shows that for any prime power $q$ there is a smooth pointless curve over $\F_q$ of arbitrary large genus $g$ divisible by $q-1$. This result is weaker than Theorem $\ref{thm:linear_bound}$ unless $q=2$ (since in that case an~Artin--Schereier curve is hyperelliptic). It seems hard to obtain better results using the Artin--Schreier coverings of the projective line.
\end{remark}

\end{document}